\documentclass[letterpaper,reqno,12pt]{amsart} 
\usepackage{setspace}
\usepackage{graphicx} 
\usepackage{cite}
\usepackage{amssymb, enumerate, color}
\usepackage{hyperref}
\usepackage{enumitem}
\usepackage{graphicx}
\usepackage{subcaption}
\usepackage{here}
\usepackage[margin=1.2in]{geometry}
\usepackage{amsmath}
\usepackage{fancyhdr}
\usepackage{indentfirst}
\usepackage{tikz-cd}
\usepackage{tikz}
\usepackage{amsthm}
\usepackage[section]{placeins}
\usepackage{mathrsfs}
\usepackage[all]{xy}
\usepackage{setspace}
\usepackage{xcolor}
\usepackage[T1]{fontenc}
\usepackage[utf8]{inputenc}
\hypersetup{colorlinks=true,colorlinks=true,linkcolor=blue,urlcolor=blue, citecolor=blue}

\newcommand{\tbf}{\textbf}

\numberwithin{equation}{section} 

\theoremstyle{plain}

\newtheorem{theorem}{Theorem}[section]
\newtheorem{lemma}[theorem]{Lemma}

\newtheorem{theoremdef}[theorem]{Theorem-Definition}
\newtheorem{corollary}[theorem]{Corollary}
\newtheorem{conjecture}[theorem]{Conjecture}

\newtheorem{definition}{Definition}[section]
\newtheorem{remark}{Remark}[section]

\title{On the canonical bundle formula and effective birationality for Fano varieties in char $p>0$}
\author{Xintong Jiang}
\address{Tsinghua University}
\email{xt-jiang21@mails.tsinghua.edu.cn}
\date{}

\begin{document}
\begin{abstract}
    In this paper, we give some results on the birational geometry of varieties of Fano type and boundedness problems in positive characteristic, including a result ensuring that boundedness is stable under normalizations, a version of canonical bundle formula which is easier to use, and the effective birationality of certain weak Fano varieties with good singularities, which is predicted by the BAB conjecture.
\end{abstract}
\maketitle
\markboth{Xintong Jiang}{On the canonical bundle formula and effective birationality for Fano varieties in char $p>0$}
\tableofcontents
\section{Introduction}
We work on an algebraically closed field $k$ of characteristic $p>0$. Many aspects of birational geometry in characteristic 0 are well understood, including a special case of MMP\cite{bchm}, the boundedness of certain varieties\cite{birkar2019antipluricanonicalsystemsfanovarieties}\cite{birkar2020singularitieslinearsystemsboundedness}, acc property for log canonical thresholds\cite{hacon2012acclogcanonicalthresholds} and canonical bundle formulas\cite{2021arXivpositivityofmoduli}. However, in positive characteristic, very few results are known, and most of them are up to dimension 3 even the resolution of singularities, which is very crucial for the study of birational geometry. The most hard part essentially comes from the failure of vanishing theorems and the lack of cognition of very special morphisms that appear only in positive characteristic like inseparable morphisms which destroy covering lemmas. In recent years, many positive characteristic results have been proved thanks to the development of F-singularities, including the MMP in dimension 3\cite{threedimmmp}\cite{birkar2014existenceflipsminimalmodels}\cite{birkar2014existencemorifibrespaces}, some weak versions of canonical bundle formulas\cite{Witaszek2017OnTC}\cite{benozzocanonicalbundleformula}, F-adjunctions\cite{schwede2009fadjunction}, and some weak versions of boundedness of certain varieties\cite{Das_2019}\cite{zhuang2020fanovarietieslargeseshadri}\cite{sato2024boundednessweakfanothreefolds}. A famous conjecture in birational geometry is the BAB conjecture, which predicts the boundedness of Fano varieties:
\begin{conjecture}[BAB Conjecture]
    Let $d$ be a natural number and $\epsilon>0$ be a real number, then the set of $\epsilon$-lc Fano varieties of dimension $d$ forms a bounded family.
\end{conjecture}
The condition $\epsilon$-lc cannot be strengthened into klt, c.f.\cite[1.2]{birkar2020singularitieslinearsystemsboundedness}. This conjecture was proved in characteristic 0 by Birkar in \cite{birkar2019antipluricanonicalsystemsfanovarieties}\cite{birkar2020singularitieslinearsystemsboundedness} using Shokurov's theory of complements. In positive characteristic, the conjecture was proved in dimension 2 by Alexeev\cite{boundednessandk2} and in toric cases by the Borisov brothers\cite{Borisov1993SINGULARTF}. In dimension 3, the conjecture is indeed widely open even for effective birationality, there are few results including when the Gorenstein index is bounded and some other cohomological conditions are added\cite{sato2024boundednessweakfanothreefolds} and etc. The boundedness of the volume with bounded Gorenstein index is shown in \cite{Das_2019}. Though very few are known, and even many results of birational geometry in positive characteristic fails, this conjecture is still supposed to be true. In this paper, we study some intermediate results of this conjecture for Fano threefolds in positive characteristic.

Our first observation is that boundedness is stable under normalizations, which is not seen before in positive characteristic.
\begin{theorem}[Theorem 2.4]
    Let $k$ be an algebraically closed field, suppose that a family $\mathcal P$ of projective varieties $k$ is bounded, then the set $\mathcal P^\nu$ of normalization of the elements in $\mathcal P$ is also bounded.
\end{theorem}

The next result is a more useful canonical bundle formula for Fano type fibrations up to dimension 3:
\begin{theorem}[Canonical bundle formula for threefolds, Theorem 3.3]
    Assume $(X,B)$ is a klt pair over an algebraically closed field of characteristic $p$, $f:X\to Z$ is a contraction with dim $X=3$ and dim $Z>0$, $K_X+B\sim_\mathbb Q 0/Z$, $B$ is relatively big, then we have:
    \begin{enumerate}
        \item if dim $Z=2$, let $\Phi\subset [0,1]\cap \mathbb Q$ be a DCC set. Suppose $B\in\Phi$, then there is a prime number $p_0=p_0(\Phi)$, such that for all $p>p_0$, the canonical bundle formula holds for $(X,B)/Z$.
        \item if dim $Z=1$ and general fibers of $X/Z$ are normal, let $I>0$ be a natural number. Then there is a $p_0=p_0(I)$ such that for all $p>p_0$, suppose $I(K_X+B)$ is Cartier near the generic fiber, the canonical bundle formula holds for $(X,B)/Z$.
        \item if dim $Z=1$ and general fibers of $X/Z$ are normal, let $\epsilon>0$ be a real number and  $\Phi\subset [0,1]\cap \mathbb Q$ be a finite set. Then there is a prime number $p_0=p_0(\epsilon,\Phi)$ such that for all $p>p_0$, suppose $B\in\Phi$ and $(X,B)$ is $\epsilon$-lc,  the canonical bundle formula holds for $(X,B)/Z$.
    \end{enumerate}
\end{theorem}
With the canonical bundle formula, a priori, non-klt center adjunction can be developed, which suggests a tool to construct complements from the non-klt centers, which is a similar proposition of \cite[3.6, 4.9]{Das_2016}:
\begin{corollary}[Non-klt center adjunction, Theorem 3.5,3.6]
    Assume $(X,B)$ is a projective klt pair of dimension $3$ over an algebraically closed field $k$, $G\subset X$ a subvariety with normalization $F$, $X$ is $\mathbb Q$-factorial near $G$, $\Delta\geq 0$ a $\mathbb Q$-Cartier divisor and $(X,B+\Delta)$ is lc near $G$ and there is a unique non-klt place of this pair whose center is $G$. Let $\Phi$ be a finite set of rationals in $(0,1]$ such that $B\in\Phi$. Then there are a $p_0=p_0(\Phi)\in\mathbb N$ and a DCC set $\Psi=\Psi(\Phi)$ such that for all char$(k)>p_0$, there is a $\mathbb Q$-divisor $\Theta_F\in[0,1]$ such that the following hold:
    \begin{enumerate}
        \item There is $P_F$ on $F$ such that 
        \[K_F+\Theta_F+P_F\sim_\mathbb Q (K_X+B+\Delta)|_F,\]
        here $\Theta_F$ is well defined and $P_F$ is determined up to $\mathbb Q$-linear equivalence and $P_F$ is pseudo-effective.
        \item $\Theta_F\in\Psi$.
        \item $M\geq 0$ is a $\mathbb Q$-Cartier divisor on $X$ with coefficients $\geq 1$ and $G$ is not a subset of $\text{Supp}M$, then for every component $D$ of $M_F:=M|_F$, we have $\mu_D(\Theta_F+M_F)\geq 1$.
    \end{enumerate}
\end{corollary}
With the positivity of the moduli part, one can also see that a contraction of a Fano type variety is also Fano type under certain conditions.
\begin{corollary}[Corollary 3.7]
    Assume $X$ is a threefold of Fano type, $X\to Z$ is a contraction, and dim $Z>0$ with normal general fibers. Suppose there is an ample $\mathbb Q$-divisor $A$ on $Z$ and some given $I>0$ such that $I(K_X+B)\sim -If^*A$ are Cartier divisors for some klt pair $(X,B)$, then there is a $p_0=p_0(I)$ such that for every $p>p_0$, such $Z$ is of Fano type.
\end{corollary}
Besides the results above, the next result in this paper is the effective birationality for Fano varieties with a bounded Gorenstein index and good F-singularities:
\begin{theorem}[Effective birationality, Theorem 4.4]
Let $d,I\in\mathbb N$ be natural numbers where $I$ is not a multiple of $p$, then there is $n,m\in\mathbb N$ depending only on $d,I$, such that if $X$ is a strongly F-regular weak Fano variety of dimension $d$ and $IK_X$ is a Cartier divisor, then $vol(-nK_X)>(2d)^d$ and $|-mK_X|$ defines a birational map.
\end{theorem}
As a corollary, if the condition of boundedness of log discrepancy is replaced by that of F-signature, which is a quantity that describes the wildness of the Frobenius map, i.e. F-singularities, which is thought to be a closed analogue of MMP singularities in positive characteristic, then the effective birationality holds:
\begin{corollary}[Corollary 4.5]
    Assume $X$ is $\epsilon$-strongly F-regular weak Fano variety of dimension $d$ and the Gorenstein index is not divisible by $p$, then there is a constant $m=m(d,\epsilon)$ such that $|-mK_X|$ defines a birational map.
\end{corollary}

Here we give the outline of the whole paper: In Chapter 2, we introduce preliminary knowledge we need in the paper, also we prove Theorem 1.2; in Chapter 3, we prove Theorem 1.3, Corollary 1.4 and Corollary 1.5 and introduce the F-adjunctions; in Chapter 4 we prove Theorem 1.6 and Corollary 1.7. Many of the ideas in this article are modified and retrofitted from \cite{birkar2019antipluricanonicalsystemsfanovarieties} and some other papers on birational geometry in positive chracteristic. If some work’s proof is absolutely characteristic free but the known reference only proves the 0-chracteristic case, we will refer it directly. We believe that there should be further and deeper work on the boundedness of algebraic varieties in positive characteristic in the future.

I would like to thank my supervisor Caucher Birkar for his patient guidance and suggestions. I also would like to thank Mao Sheng, Paolo Cascini, Jihao Liu, Fulin Xu, Marta Benozzo, Xiaowei Jiang for useful discussions and comments while writing this paper.
\section{Preliminaries}
In this section, some basic notions and results of birational geometry are mentioned for readers who are not familiar with. In this paper, all varieties are quasi-projective andreduced schemes of finite type over an algebraically closed field $k$ of characteristic $p>0$ and the ambient variety $X$ is projective unless stated otherwise. 
\paragraph{\tbf{Resolution of singularities}}
For a variety $X$, a resolution of singularity is a proper birational map $f:Y\to X$ from a smooth variety $Y$, which is an isomorphism on the regular locus of $X$, and for the singular locus $X_{sing}$, one has $f^{-1}X_{sing}$ is a divisor with simple normal crossings. 
\begin{theorem}
    For a 3-dimensional variety $X$, there is a resolution of singularity $f:Y\to X$ which is obtained by a sequence of blow-ups along smooth centers over $X_{sing}$.
\end{theorem}
\begin{proof}
    See \cite{Cutkosky2004ResolutionOS}, \cite{Cossart2008ResolutionOS} and \cite{Cossart2009RESOLUTIONOS}.
\end{proof} 
\paragraph{\tbf{Divisors, pairs and linear systems}} We define $\mathbb N=\mathbb Z^{\geq0}$. Divisors refer to Weil divisors, which correspond to reflexive sheaves of rank 1 up to linear equivalence. The notion of Cartier divisors follows the usual definition, which correspond to invertible sheaves or line bundles up to linear equivalence. For a variety $X$, the dualizing sheaf $\omega_X$ denotes to be the lowest cohomology sheaf of its dualizing complex. $\omega_X$ is a reflexive sheaf, which corresponds to a divisor $K_X$ up to linear equivalence, which is called the canonical divisor of $X$. 

For $\mathbb F=\mathbb Q,\mathbb Z_{(p)}=(\mathbb Z-p\mathbb Z)^{-1}\mathbb Z$, $\mathbb F$-divisors are the $\mathbb F$-linear combination of divisors and $\mathbb F$-linear equivalences between divisors are generated by $\mathbb F$-linear combination of linear equivalences. A similar definition applies for Cartier divisors. An $\mathbb F$-divisor is $\mathbb F$-Cartier if it is $\mathbb F$-linear equivalent to an $\mathbb F$-Cartier divisor. $(X,B)$ is called a sub pair if $X$ is a normal variety and $B$ is a $\mathbb Q$-divisor such that $K_X+B$ is $\mathbb Q$-Cartier and $B\leq 1$ (in coefficients).  A sub pair $(X,B)$ is called a pair if $B\geq 0$. Here, $B$ is called a boundary if $(X,B)$ is a pair.

For an $\mathbb F$-divisor $M$, we often denote $H^0(M):=H^0(X,\mathcal O_X(\lfloor M\rfloor))$. The linear system $|M|=Proj(H^0(M))=\{N\sim M,N\geq0\}$. The $\mathbb F$-linear system is defined as $|M|_\mathbb F:=\{N\sim_\mathbb F M,N\geq0\}$, in particular $|M|_\mathbb Q=\bigcup\limits_{m\in\mathbb N} \frac{1}{m}|mM|$. The base locus $Bs(|M|)$ denotes the maximal closed subset of $X$ contained in each $N\in|M|$, the stable base locus is defined as $Bs(|M|_\mathbb Q):=\bigcap\limits_{m\in\mathbb N} Bs(|mM|)$. If $|M|\neq\emptyset$, $|M|$ will define a rational map $\phi_M:X\dashrightarrow |M|^\lor\simeq \mathbb P^n$, which is determined on $X-Bs(|M|)$ by mapping $x$ to the hyperplane in $H^0(M)$ whose elements are those $N\sim M$ which passes $x$. 

\paragraph{\tbf{B-divisors and generalized pairs}}

A b-divisor $\mathbf M$ over $X$ is an $\mathbb Q$-Cartier divisor $M_Y$ on a birational model $f:Y\to X$ projective over $X$, up to the equivalence generated by pull-backs along birational models projective over $X$. Set $M:=f_*M_Y$, one often uses $M$ to represent the b-divisor $\mathbf M$ for convenience.

A generalized pair is given as $(X',B'+M')/Z$ where $X'$ is a normal variety with a projective morphism $X'\to Z$, $B'\geq0$ an $\mathbb Q$-divisor (usually $B'\leq 1$) on $X'$ and a b-$\mathbb Q$-Cartier $b$-$\mathbb Q$-divisor $M'$ represented by some projective birational morphism $\phi:X\to X'$ and an $\mathbb Q$-Cartier $\mathbb Q$-divisor $M$ on $X$ such that $M$ is nef over $Z$ and $M'=\phi_*M$ and $K_{X'}+B'+M'$ is $\mathbb Q$-Cartier. Since $M'$ is defined birationally, one may assume that $X\to X'$ is a log resolution. $M$ is viewed as a b-divisor in generalized pairs.

\paragraph{\tbf{Singularities from MMP}}
For a prime divisor $D$ on a log resolution $W/X$ of the (resp. sub-) pair $(X,B)$, let $K_W+B_W$ be the pullback of $K_X+B$, the log discrepancy $a(D,X,B)$ defines to be $1-\mu_DB_W$. The log discrepancy is a number defined up to strict transformations along the birational maps between smooth models of $X$. One say the (resp. sub-)pair $(X,B)$ is (resp.sub-)lc (resp. klt) (resp. plt) (resp. canonical) (resp. terminal) (resp. $\epsilon$-lc) if $a(D,X,B)\geq 0$ (resp. $>0$) (resp. $>0$ for exceptional $D$) (resp. $\geq 1$ for exceptional $D$) (resp. $>1$ for exceptional $D$) (resp. $\geq\epsilon$) for every $D$. A non-klt place of a sub pair $(X,B)$ is a prime divisor $D$ on birational models of $X$ such that $a(D,X,B)\leq0$. A non-klt center is the image on $X$ of a non-klt place. A (resp.sub-)pair is (resp.sub-)dlt if it is lc and log smooth near generic points of non-klt centers.

For a generalized pair $(X',B'+M')$ and a divisor $D$ over $X$, take a sufficiently high resolution $f:X\to X'$ defining $M'=f_*M$ and contains $D$, we define $K_X+B+M:=f^*(K_{X'}+B'+M')$, and one can similarly define generalized version of lc, klt, plt, $\epsilon$-lc by considering the generalized log discrepancy $a(D,X',B'+M'):=1-\mu_D(B)$.
\paragraph{\tbf{Contractions and minimal model programs}}
An algebraic fiber space or a contraction $X/Y$ is a projective morphism $f:X\to Y$ between varieties such that $f_* \mathcal O_X=\mathcal O_Y$, which is equivalent to a projective surjective morphism such that the finite part of the Stein factorization is trivial, or a projective surjective morphism such that the function field of $Y$ is algebraically closed in that of $X$. It's well known the fibers are connected.

We use standard results of the minimal model program (MMP), MMP in char $k>5$ up to dimension 3 is already fully known:
\begin{theorem}[c.f. \cite{birkar2014existenceflipsminimalmodels}\cite{birkar2014existencemorifibrespaces}]
    Let $(X,B)/Z$ be a 3-dimensional klt pair over $k$ of char $>5$, $X\to Z$ be a projective contraction, then there is a minimal model program$/Z$ on $K_X+B$ such that:
    \begin{enumerate}
        \item If $K_X+B$ is pseudo-effective$/Z$, then the MMP ends with a log minimal model$/Z$.
        \item If $K_X+B$ is not pseudo-effective$/Z$,then the MMP ends with a Mori fiber space$/Z$.
    \end{enumerate}
\end{theorem}
\paragraph{\tbf{$\mathbb Q$-factorialization}}
A normal variety is $\mathbb Q$-factorial if every divisor is $\mathbb Q$-Cartier. For a generalized pair $(X',B',M')$ with data $X\to X'$, a $\mathbb Q$-factorial generalized dlt model is a $\mathbb Q$-factorial generalized dlt generalized pair $(X'',B''+M'')$ with a projective birational morphism $\psi:X''\to X'$ under a log resolution $X\to X''$ (after taking a common resolution) such that $B''$ and $M''$ are pushdowns of $B$ and $M$, in particular $K_{X''}+B''+M''=\psi^*(K_{X'}+B'+M')$, and if every exceptional prime divisor of $\psi$ appears in $B''$ with coefficient 1. Such model exists for generalized lc pairs. If $(X',B'+M')$ is generalized klt, then there is a $\mathbb Q$-factorial generalized klt model and $\psi$ is a small morphism (i.e. no divisor is contracted or extracted), this is called a small $\mathbb Q$-factorialization.
\paragraph{\tbf{Volumes, Kodaira dimensions and Iitaka fibrations}}
Let $X$ be a normal projective variety of dimension $d$ and $D$ a $\mathbb Q$-divisor on $X$. We define the Kodaira dimension $\kappa(D)$ (resp. the numerical Kodaira dimension $\kappa_\sigma(D)$) to be $-\infty$ if $D$ is not effective (resp. pseudo-effective), and to be the largest integer $r$ such that $\limsup_{m\to\infty}\frac{h^0(\lfloor mD\rfloor)}{m^r}>0$ (resp. $\limsup_{m\to\infty}\frac{h^0(\lfloor mD\rfloor+A)}{m^r}>0$ for some very ample divisor $A$). We define the volume $vol(D):=\limsup_{m\to \infty}\frac{h^0(\lfloor mD\rfloor)}{m^d}$ and say $D$ is big if $vol(D)>0$ as usual. $|\lfloor mD\rfloor|$ will define a morphism $\phi_m=\phi_{\lfloor mD\rfloor}$. The dimension of the image of $\phi_m$ will stabilize to $\kappa(D)$. The stabilized rational fibration $\phi:X\dashrightarrow \phi(X)$ is called the Iitaka fibration. 
\paragraph{\tbf{Fano pairs and varieties of Fano type}}
Let $(X,B)$ be a pair with a contraction $X\to Z$, we say $(X,B)$ is log Fano (resp. weak log Fano) over $Z$ if $-(K_X+B)$ is ample (resp. nef and big) over $Z$. We assume $B=0$ when we don't mention $B$. We say $X$ is of Fano type over $Z$ if $(X,B)$ is klt weak log Fano over $Z$ for some boundary $B$, or equivalently if $(X,\Gamma)$ is klt and $\Gamma$ is big over $Z$ and $K_X+\Gamma\sim_\mathbb Q0/Z$.

Suppose $f:X\to Y$ is a birational contraction and $(X,B)$ is of Fano type, then $(Y,B_Y=f_*B)$ is of Fano type as the pushforward of a big divisor is big. Suppose $(X,B)\dashrightarrow(Y,B_Y)$ is a sub-crepant birational map (i.e. there is a common resolution $W$ such that $(K_X+B)|_W\leq (K_Y+B_Y)|_W$), then $(Y,B_Y)$ is of Fano type will imply that $(X,B)$ is of Fano type. Hence taking crepant resolutions, running MMP and taking $\mathbb Q$-factorializations will keep the property of Fano type.

\paragraph{\tbf{Bounded families}}
Now we introduce the notion of bounded families mentioned in the BAB conjecture. A couple $(X,D)$ is formed by a normal projective variety $X$ and a divisor $D$ on $X$ such that the coefficient of $D$ falls in $\{0,1\}$. Isomorphisms between couples are isomorphisms between base schemes such that the morphism is compatible and onto for boundaries. A set $\mathcal P$ of couples is birationally bounded (resp. bounded) if there exist finitely many projective morphisms $V^i\to T^i$ of varieties and reduced divisors $C^i$ on $V^i$ such that for each $(X,D)\in\mathcal P$ there is an $i$ and a closed point $t\in T$ and a birational isomorphism (resp. isomorphism) $\phi:V_t^i\dashrightarrow X$ such that the fiber $(V_t^i,C_t^i)$ over $t$ is a couple and $E\leq C_t^i$, where $E$ is the sum of the strict transform of $D$ and the reduced exceptional divisor of $\phi$. A set $\mathcal R$ of projective pairs $(X,B)$ is said to be log birationally bounded (resp. log bounded) if the set of $(X,\text{Supp}B)$ is birationally bounded (resp. bounded). And if $B=0$ for all the elements in $\mathcal R$, we usually remove the log and say the set is birationally bounded (resp. bounded). We offer a useful characteristic-free criterion for boundedness here, and from this one can assume that $(V_t^i,C_t^i)$ is isomorphic to $(X,D)$ for bounded families in the above definition (c.f. \cite[2.21]{birkar2019antipluricanonicalsystemsfanovarieties}).
\begin{lemma}[\cite{birkar2019antipluricanonicalsystemsfanovarieties} 2.20]
    If $\mathcal P$ is a set of couples of dimension $d$, $\mathcal P$ is bounded if and only if there is an $r\in \mathbb N$ such that for any $(X,D)\in \mathcal P$, there is a very ample divisor $A$ on $X$ such that $A^d\leq r$ and $A^{r-1}D\leq r$.
\end{lemma}
If a set of varieties $\mathcal R$ is bounded, then the Gorenstein indices, the (anti-)canonical volumes, the indices of the effective Iitaka fibrations (e.g. for varieties that $\kappa(K_X)\geq 0$, the minimal $m$ such that $|mK_X|$ defines the Iitaka fibration), the Picard numbers, etc. are all bounded. Moreover, one can prove that the normalizations of a bounded family also form a bounded family.

\begin{theorem}
    Let $k$ be an algebraically closed field, suppose that a family $\mathcal P$ of projective varieties $k$ is bounded, then the set $\mathcal P^\nu$ of normalization of the elements in $\mathcal P$ is also bounded.
\end{theorem}
\begin{proof}
    In characteristic $0$, suppose that a family $\mathcal P$ of projective varieties over an algebraically closed field is bounded and encoded in a finite set of fibrations $V_i\to Z_i$. By choosing a suitable stratification on $Z_i$'s, one may assume $\mathcal P$ is parametrized by a dense set of such $Z_i$'s and $Z_i$'s are smooth by base changing to a resolution of singularity. We claim that they can be placed in general positions by taking a finer stratification in the following sense: the corresponding fibers of the normalization $V_i^\nu\to Z_i$ are exactly the normalization of the initial fibers. In fact, suppose that $F=V_{i,z}$ is a fiber over a closed point $z\in Z$, then the composition map $F^\nu\to F\to V_i/Z_i$ will factor through $V_i^\nu$ by the definition of normalization. This will give a map $F^\nu\to (V_i^\nu)_z$. However, since $V_i^\nu$ is normal and the general fibers are also normal since we work in characteristic $0$ and the parametrizing fibers are in general positions. Hence we have a canonical map $(V_i^\nu)_z\to F^\nu$ by definition and the two maps give the identity which is what we need. As a corollary for such family $\mathcal P$, the set $\mathcal P^\nu$ of normalization of the elements in $\mathcal P$ is also bounded. In characteristic $p$, the same will hold thanks to the technique of Frobenius base changes:
    
    Suppose $\kappa$ is a field of charcateristic $p>0$ and $X/\kappa$ is a normal projective variety with $H^0(X,\mathcal O_X)=\kappa$. Set $F^e:\kappa\to \kappa$ to be the $e$-th iteration of the Frobenius morphism and set $X_e$ to be the normalization of the reduced scheme of the base change of $X$ under $F_e$, i.e. we have the following pullback diagram in the category of normal schemes:
\[\begin{tikzcd}
	{X_e} & X \\
	\kappa & \kappa
	\arrow[from=1-1, to=1-2]
	\arrow[from=1-1, to=2-1]
	\arrow[from=1-2, to=2-2]
	\arrow["{F^e}", from=2-1, to=2-2]
\end{tikzcd}\]
We claim that for $e$ large enough, $X_e$ is geometrically normal and moreover when base change to the algebraic closure, $\bar{X_e}$ is isomorphic to the normalization of $\bar X$. Indeed set $\kappa_e=\kappa^{\frac{1}{p^e}}$, one can equip the canonical embedding $\iota_e:\kappa\to\kappa_e$, set $X^e$ to be the normalization of the reduced scheme of the base change of $X$ under $\iota_e$. We see $\alpha_e:\kappa_e\to\kappa$ which sends $x$ to $x^p$ is an injective and surjective ring map, hence an isomorphism. We have the following commutative pullback diagram in the category of normal schemes:
\[\begin{tikzcd}
	& {X^e} \\
	{X_e} &&& X \\
	& {\kappa_e} \\
	\kappa &&& \kappa
	\arrow[from=1-2, to=2-4]
	\arrow[from=1-2, to=3-2]
	\arrow[from=2-1, to=1-2]
	\arrow[from=2-1, to=2-4]
	\arrow[from=2-1, to=4-1]
	\arrow[from=2-4, to=4-4]
	\arrow["{\iota_e}", from=3-2, to=4-4]
	\arrow["{\alpha_e}", from=4-1, to=3-2]
	\arrow["{F^e}"', from=4-1, to=4-4]
\end{tikzcd}\]
Here $\alpha_e^X:X_e\to X^e$ is an isomorphism of schemes. When passing to the algebraic closure, one may first pass to the perfect closure $\kappa_\infty=\bigcup\limits_{e=0}^\infty\kappa_e$. $\kappa_\infty$ is perfect since the Frobenius $\kappa_\infty\xrightarrow{F}\kappa_\infty$ is a bijection. If the claim holds when passing to the perfect closure, then it holds when passing to the algebraic closure since being normal is equivalent to being both $S_2$ and $R_1$, where the first condition is invariant under fppf base change and the second is invariant under étale base change. Hence we have the following picture:
\[\begin{tikzcd}
	& {X^\infty} &&& {X^{e+1}} && {X^e} \\
	{X_e^\infty} &&& {X_e^1} && {X_e} &&& X \\
	& {\kappa_\infty} &&& {\kappa_{e+1}} && {\kappa_e} \\
	{\kappa_\infty} &&& {\kappa_1} && \kappa &&& \kappa
	\arrow["\cdots"{description}, from=1-2, to=1-5]
	\arrow[from=1-2, to=3-2]
	\arrow[from=1-5, to=1-7]
	\arrow[from=1-5, to=3-5]
	\arrow[from=1-7, to=2-9]
	\arrow[from=1-7, to=3-7]
	\arrow[from=2-1, to=1-2]
	\arrow["\cdots"{description}, from=2-1, to=2-4]
	\arrow[from=2-1, to=4-1]
	\arrow[from=2-4, to=1-5]
	\arrow[from=2-4, to=2-6]
	\arrow[from=2-4, to=4-4]
	\arrow[from=2-6, to=1-7]
	\arrow[from=2-6, to=2-9]
	\arrow[from=2-6, to=4-6]
	\arrow[from=2-9, to=4-9]
	\arrow["\cdots"{description}, from=3-2, to=3-5]
	\arrow[from=3-5, to=3-7]
	\arrow["{\iota_e}", from=3-7, to=4-9]
	\arrow["{F^e}","{\simeq}"', from=4-1, to=3-2]
	\arrow["\cdots"{description}, from=4-1, to=4-4]
	\arrow[from=4-4, to=3-5]
	\arrow[from=4-4, to=4-6]
	\arrow["{\alpha_e}", from=4-6, to=3-7]
	\arrow["{F^e}"', from=4-6, to=4-9]
\end{tikzcd}\]
One will get $X_e^\infty\to X^\infty$ is also an isomorphism of $\kappa^\infty$-varieties. Hence one only need to prove the claim for $X^e$ rather than $X_e$. Indeed we may work in affine settings and consider the conductor ideal $\mathfrak I_e$ of $X^\infty\to  X^e_{\bar\kappa}$ we get an infinitely long accending chain $\mathfrak I_0\subset\mathfrak I_1\subset\mathfrak I_2\subset\cdots$. Since $X$ is noetherian, one sees $\mathfrak I_e$ stablizes finally, say for $e\geq N$. Which means for any $e>N$, $X^{e+1}\to X^e_{\kappa_{e+1}}$ will have conductor ideal the whole ring, hence $X^e_{\kappa_{e+1}}$ itself is normal over $\kappa_{e+1}$. So one see $X_e$ is normal after base change to arbitrary finite field extension by induction and hence geometrically normal. Moreover $X^\infty\to X^e_{\kappa_\infty}$ will give an isomorphism of $\kappa^\infty$-schemes as desired by the definition of a normalization. So the claim is proved.

Now we return to our lemma. By taking irreducible components of $X\in\mathcal P$, one may assume that all $X\in\mathcal P$ are integral. Since $\mathcal P$ is bounded, there is a finite set of fibrations $\mathcal X_i\to\mathcal Z_i$ such that for each $X\in\mathcal P$, $X$ is a fiber of some $\mathcal X_i/\mathcal Z_i$ over a closed point. One may assume that there is only one morphism of varieties $\mathcal X\to \mathcal Z$ without loss of generality. If there is some open subvariety $\mathcal U\subset \mathcal Z$ such that no point in $\mathcal U$ gives a fiber $X\in\mathcal P$, we just cut it off from $\mathcal Z$ and replace $\mathcal Z$ by the irreducible (reduced) components of $\mathcal Z-\mathcal U$ with the corresponding $\mathcal X$ the base change of the initial one along $\mathcal Z-\mathcal U\to \mathcal Z$. The process stops after finitely many steps by noetherian induction. Hence, we may assume that there is a dense set $\mathcal W\subset \mathcal Z$ parametrizing $X\in\mathcal P$. Since there is a dense set of fibers $F$ such that $F$ are integral projective varieties, the geometric generic fiber is also integral, and in particular, one may assume $\mathcal X\to \mathcal Z$ is a contraction.

Now consider the $e$-th iterated absolute Frobenius morphism $F^e:\mathcal Z\to \mathcal Z$, which is identity on the topological spaces. Set $\mathcal X_e$ to be the normalization of the fiber product $\mathcal X\times_{\mathcal Z,F^e}\mathcal Z$. For any closed fiber $F$ of $\mathcal X/\mathcal Z$, consider the corresponding closed fiber $F_e$ of $\mathcal X\times_{\mathcal Z,F^e}\mathcal Z\to\mathcal Z$, that is , there is some $t=Spec (k)\to \mathcal Z$ such that $F_e=(\mathcal X\times_{\mathcal Z,F^e}\mathcal Z)_t$. Consider the following pullback diagram in the category of schemes:
\[\begin{tikzcd}
	{F_e} && F \\
	& {\mathcal X\times_{\mathcal Z,F^e}\mathcal Z} && {\mathcal X} \\
	{Spec(k)} && {Spec(k)} \\
	& {\mathcal Z} && {\mathcal Z}
	\arrow[from=1-1, to=1-3]
	\arrow[from=1-1, to=2-2]
	\arrow[from=1-1, to=3-1]
	\arrow[from=1-3, to=2-4]
	\arrow[from=1-3, to=3-3]
	\arrow[from=2-2, to=2-4]
	\arrow[from=2-2, to=4-2]
	\arrow[from=2-4, to=4-4]
	\arrow["{F^e}", from=3-1, to=3-3]
	\arrow[from=3-1, to=4-2]
	\arrow[from=3-3, to=4-4]
	\arrow["{F^e}", from=4-2, to=4-4]
\end{tikzcd}\]
Since $k$ is algebraically closed, $F^e:Spec(k)\to Spec(k)$ is an isomorphism and hence $F_e$ is isomorphic to $F$ as $k$-varieties. On the other hand, for any contraction $V/T$ where $T$ is normal and integral, the generic fiber of the normalization $V^\nu/T$ is the normalization of the generic fiber of $V/T$ as $k(T)$-varieties, since the generic fiber of a normal whole space is also normal and the rest can be checked by the functoriality of normalizations. As a result, the generic fiber of $\mathcal X_e/\mathcal Z$ is just $(\mathcal X_{\eta_\mathcal Z}\times_{k(\mathcal Z),F^e}k(\mathcal Z))^\nu$. By our claim for some $e$ sufficiently large it is geometrically normal. So, over some open subset $\mathcal U\subset\mathcal Z$, all the closed fibers of $\mathcal X_e/\mathcal U$ are normal since geometric normality (moreover geometric regularity and geometric reducedness) is an open condition (c.f. \cite[II.6.9.1 and III.12.1.1]{PMIHES_1967__32__5_0}) and $k$ is algebraically closed. Combined with all the results above and using the same arguments in chracteristic $0$, one sees that any $X\in\mathcal P$ which is parametrized by some $t\in\mathcal U$, $X^\nu$ is isomorphic to $ (\mathcal X^e)_t$. Hence, the normalizations of such $X$ are parametrized by a subset of $\mathcal X^e/\mathcal U$ up to isomorphisms. Since there is a dense set $\mathcal W\subset \mathcal Z$ parametrizing all $X\in\mathcal P$, cut $\mathcal U$ away from $\mathcal Z$, $\mathcal P$ decreases strictly. Moreover, the rest $X\in\mathcal P$ will be parametrized by a subset of $\mathcal X/(\mathcal Z-\mathcal U)$. By noetherian induction, the assertion holds.
\end{proof}
\begin{corollary}
    Suppose $k$ is an arbitrary field and $\mathcal P$ is a bounded set of projective varieties$/k$ with $H^0(X,\mathcal O_X)=k$ for all $X\in \mathcal P$. Then the set consisting of the normalizations of such $X_{\bar k}$ for all $X\in\mathcal P$ is bounded.
\end{corollary}
\begin{proof}
    By Lemma 2.4, one only needs to see that the set of $X_{\bar k}$'s is bounded where $X\in\mathcal P$. Suppose $X\in \mathcal P$ are parametrized by a closed point $t\in\mathcal Z$ via a fibration $\mathcal X\to \mathcal Z$ by boundedness. Then $t=Spec(k)$ since $H^0(X,\mathcal O_X)=k$. Hence $X_{\bar k}$ is parametrized by $\bar t\in\mathcal Z_{\bar k}$ via $\mathcal X_{\bar k}\to \mathcal Z_{\bar k}$ and the assertion follows.
\end{proof}
\paragraph{\tbf{F-singularities}} Other than singularities from MMP (e.g klt,lc,etc.), there is another system of singularities in positive characteristic, by detecting the wildness of the Frobenius map. The Frobenius $F:X\to X$ is identity on the base space and the morphism on the structure sheaf is defined by the inclusion $F^\sharp:\mathcal O_X\to F_*\mathcal O_X=\mathcal O_X^{\frac{1}{p}}$ (as X is reduced). The inspiration is the theorem of Kunz\cite{kunz1969regularlocal}, which says that a local k-ring essentially of finite type is regular if and only if the Frobenius map is free. Later Huneke and Hotchster stated the theory of F-singularities, one can see a detailed note on this topic in\cite{hotchstertightclosure}. Here we give the brief definitions:
\begin{definition}[Map pairs]
    Let $(X,\Delta)$ be a log pair which is $\mathbb Z_{(p)}$-log Gorenstein for simplicity. Assume $(p^e-1)(K_X+\Delta)$ is Cartier, set a certain line bundle:
    \[\mathcal L_{e,\Delta}:=\mathcal O_X((1-p^e)(K_X+\Delta))\]  
    Take any open affine chart $U\subset X$, consider 
    \[\begin{aligned}
    \mathcal Hom_{\mathcal O_U}(F^e_*\mathcal O_U((p^e-1)\Delta),\mathcal O_U)&=\mathcal Hom_{\mathcal O_U}(F^e_*\mathcal O_U((p^e-1)\Delta)\otimes\mathcal O_U(K_U),\mathcal O_U(K_U))\\&=\mathcal Hom_{\mathcal O_U}(F^e_*\mathcal O_U((p^e-1)\Delta+p^eK_U),\mathcal O_U(K_U))\\&=(\mathcal Hom_{\mathcal O_U}(\mathcal O_U((p^e-1)\Delta+p^eK_U),\mathcal O_U(K_U)))^{\frac{1}{p^e}}\\&=F^e_*\mathcal O_U((1-p^e)(K_U+\Delta))=F^e_*\mathcal L_{e,\Delta}
    \end{aligned}
    \]
    Hence any section of $\mathcal L_{e,\Delta}$ will correspond to a map $\psi:F^e_*\mathcal O_U((p^e-1)\Delta)\to \mathcal O_U$, i.e. every $\mathbb Q$-divisor $\Delta'$ such that $D_\psi:=(1-p^e)(K_X+\Delta')\sim (1-p^e)(K_X+\Delta)$ and $D_\psi$ is effective will correspond to a non-zero map $\psi:F^e_*\mathcal O_U((p^e-1)\Delta)\to \mathcal O_U$. Dually speaking, there is a unique map $\phi_\Delta:F^e_*\mathcal L_{e,\Delta}\to\mathcal O_U$ that  corresponds to the natural map $\mathcal O_U\to F^e_*\mathcal O_U\to F^e_*(\mathcal O_U((p^e-1)\Delta))$ induced by the $\mathbb Q$-divisor $\Delta$. We call $(X,\phi_\Delta)$ the map pair that corresponds to the pair $(X,\Delta)$.
\end{definition}
\begin{definition}[Test ideal and non-F-pure ideal, F-singularities]
    Let $(X,\Delta)$ be a log pair which is $\mathbb Z_{(p)}$-log Gorenstein for simplicity. The test ideal $\tau(X,\Delta)$ (resp. non-F-pure ideal $\sigma(X,\Delta)$) of $(X,\Delta)$ is the unique smallest non-zero ideal (resp. largest ideal) $J\subset\mathcal O_X$ such that $\phi_\Delta\circ F^e_*(J\mathcal L_{e,\Delta})=J$. $(X,\Delta)$ is called strongly F-regular (resp. sharply F-pure) if $\tau(X,\Delta)=\mathcal O_X$ (resp. $\sigma(X,\Delta)=\mathcal O_X$). The closed subscheme corresponding to the ideal sheaf $\tau(X,\Delta)$ is called the non-F-regular center of $(X,\Delta)$.
\end{definition}
\begin{remark}
    There is a different but much more concise definition of F-singularities without using test ideals (or non-F-pure ideals). Again, let $(X,\Delta)$ be a log pair which is $\mathbb Z_{(p)}$-log Gorenstein for simplicity. Suppose $(p^e-1)(K_X+B)$ is Cartier near $x\in X$, we say $(X,\Delta)$ is sharply F-pure (resp. strongly F-regular) at $x$ if the natural morphism $\mathcal O_X\to F^{ne}_* \mathcal O_X\to F^{ne}_*\mathcal O_X((p^{ne}-1)\Delta)$ splits when localizing to $x$ as $\mathcal O_{X,x}$-modules for some $n>0$ (resp. for every effective Cartier divisor $D$, $\mathcal O_X\to F^{ne}_* \mathcal O_X\to F^{ne}_*\mathcal O_X((p^{ne}-1)\Delta+D)$ splits when localizing to $x$ as $\mathcal O_{X,x}$-modules for some $n>0$).  Moreover, a warning is that F-singularities doesn't behave well under MMP or resolution of singularities.
\end{remark} 

\begin{definition}[F-signatures]
    Assume $R$ a local ring over $k$, $\Delta$ is a $\mathbb Q$-divisor on Spec$(R)$, an inclusion $R^{\oplus m}\subset F^e_*R$ presents an $(R,\Delta)$ direct summand if the map is an inclusion of $R$-modules and each component's projection $F^e_*R\to R$ factors through $F^e_*R(\lceil (p^e-1)\Delta\rceil)$, denote $ a^\Delta_e$ the $e$-th F-splitting number to be the maximal $m$ such that $R^m$ is an $(R,\Delta)$ direct summand of $F^e_*R$. Suppose $R$ is of dimension $d$, we denote
    \[s(R,\Delta)=\lim\limits_{e\to\infty}\frac{a^\Delta_e}{p^{ed}}\] to be the F-signature of $(R,\Delta)$. We say a local ring pair $(R,\Delta)$ is $\epsilon$-F-regular if $s(R,\Delta)>\epsilon$, a variety is $\epsilon$-F-regular if it is $\epsilon$-F-regular at all closed points.
\end{definition}
\begin{remark}
    The F-signature can be defined defined for rings over $F$-finite fields. A ring $R$ is regular (resp. a pair $(R,\Delta)$ is F-regular) if and only if $s(R)=1$ (resp. $s(R,\Delta)>0$). Moreover, for klt pair $(R,\Delta)$, one has $s(R,\Delta)\leq mld(R,\Delta)$\cite{blickle2014fsingularitiesalterations}, hence if $(X,\Delta)$ is $\epsilon$-F-regular then it is $\epsilon$-klt. Moreover, if $(X,\Delta)$ is strongly F-regular (resp. sharply F-pure) then it is klt (resp. lc) \cite{hara2000fregularfpureringsvs}. F-signatures are similar quantities to describe F-singularities like log discrepancies for singularities from MMP. F-singularities are similar but much stronger analogues of the MMP singularities in positive characteristic.
\end{remark}
\section{Adjunctions and F-adjunctions}
Adjunctions are important tools in birational geometry, the first kind of adjunctions is the divisorial adjunction. In characteristic 0, the smooth case of divisorial adjunction is well known as adjunction formulas in complex geometry. Generally speaking, divisorial adjunctions gives a rough relation between the singularities of the ambient variety and that of the divisors:
\begin{theorem}[\cite{birkar2014existenceflipsminimalmodels} 4.1,4.2]
    Let $(X,B)$ be a pair, $S$ be a component of $\lfloor B\rfloor$, and $S^\nu\to S$ be the normalization. Then there is a canonically determined R-divisor $B_{S^\nu} \geq 0$ such that
\[K_{S^\nu} + B_{S_\nu} \sim_\mathbb Q (K_X + B)|_S\]
Moreover let $\Phi\subset [0, 1]$ be a DCC set of rational numbers, assume that:
\begin{enumerate}
    \item $(X,B)$ is lc outside a codimension 3 closed subset, and
    \item the coefficients of $B$ are in $\Phi$.
\end{enumerate}
Then $B_{S^\nu}$ is a boundary with coefficients in $\mathfrak S_\Phi$, here
\[\mathfrak S_\Phi=\{\frac{m-1}{m}+\sum\frac{l_ib_i}{m}\leq 1|m\in\mathbb Z^{>0}\cup\{\infty\},l_i\in\mathbb Z^{\geq 0},b_i\in\Phi\}.\]
\end{theorem}
Another useful adjunction is the fiber space adjunction, or more usually called the canonical bundle formula:
\begin{conjecture}[Canonical bundle formula]
    Suppose $(X,B)/Z$ a contraction with relative dimension $>0$, where $(X,B)$ is generically lc projective pair and $K_X+B\sim_\mathbb Q 0/Z$, let $\eta=\text{Spec}(k(Z))$ to be the generic point of $Z$, then there is 
    \[B_Z:=\sum\limits_{D\text{ prime divisor on }Z}(1-lct_\eta(X,B,f^*D))D\]
     and some pseudo-effective $\mathbb Q$-b-divisor $M_Z$ such that $K_X+B_X\sim_\mathbb Q f^*(K_Z+B_Z+M_Z)$. Moreover, if $(X,B)$ is lc, then $M_X$ is a b-nef $\mathbb Q$-b-divisor, where $M_X=K_X+B_X-f^*(K_Z+B_Z)\sim_\mathbb Qf^*M_Z$
\end{conjecture}
Here, $B_Z$ is called the determinant part, which is determined uniquely, and $M_X$ is called the moduli divisor. As $M_X=f^*M_Z$, we also call $M_Z$ the moduli part of the fibration. As $f:X\to Z$ is a contraction, the b-nefness of $M_X$ is equivalent to that of $M_Z$. 

In characteristic $0$, the conjecture was completely settled by a series of works (c.f.\cite{Kawamata1997SubadjunctionOL} \cite{Ambro_2005}\cite{2021arXivpositivityofmoduli}\cite{FilsemiampleModulipart}\cite{JLX22}\cite{CHLX23}.) However, in positive characteristic, the conjecture indeed fails (c.f. \cite[3.5]{Witaszek2017OnTC}).  Some works \cite{Witaszek2017OnTC}\cite{benozzocanonicalbundleformula} have settled the case when the geometric generic fiber, or equivalently, the general fibers are lc and the base or the fibers are projective curves. With some boundedness results of Fano varieties of lower dimensions, one can prove that some klt Fano type fibrations satisfy the condition of these papers, hence the canonical bundle formula holds for these fibrations. The following is our canonical bundle formula for 3-dimensional fibrations of Fano type in large characteristic:
\begin{theorem}[Canonical bundle formula for threefolds]
    Assume $(X,B)$ is a klt pair over an algebraically closed field of characteristic $p$, $f:X\to Z$ is a contraction with dim $X=3$ and dim $Z>0$, $K_X+B\sim_\mathbb Q 0/Z$, $B$ is relatively big, then we have:
    \begin{enumerate}
        \item if dim $Z=2$, let $\Phi\subset [0,1]\cap \mathbb Q$ be a DCC set. Suppose $B\in\Phi$, then there is a prime number $p_0=p_0(\Phi)$, such that for all $p>p_0$, the canonical bundle formula holds for $(X,B)/Z$.
        \item if dim $Z=1$ and general fibers of $X/Z$ are normal, let $I>0$ be a natural number. Then there is a $p_0=p_0(I)$ such that for all $p>p_0$, suppose $I(K_X+B)$ is Cartier near the generic fiber, the canonical bundle formula holds for $(X,B)/Z$.
        \item if dim $Z=1$ and general fibers of $X/Z$ are normal, let $\epsilon>0$ be a real number and  $\Phi\subset [0,1]\cap \mathbb Q$ be a finite set. Then there is a prime number $p_0=p_0(\epsilon,\Phi)$ such that for all $p>p_0$, suppose $B\in\Phi$ and $(X,B)$ is $\epsilon$-lc,  the canonical bundle formula holds for $(X,B)/Z$.
    \end{enumerate}
\end{theorem}
\begin{proof}
    The idea of the proof is to reduce to the case where the general fibers are lc and the base or the fiber are smooth curves, and then apply the known results.
    
    When dim $Z=2$. In this case, the general fibers are curves. Replacing $X$ by its resolution $X'$, we have a crepant model $(X',B'=B^+-B^-)\to (X,B)$, where $(X',B')$ is sub-lc. Also, replace $Z$ by its smooth locus, one may assume that $f:(X,B)\to Z$ is a fibration between smooth quasi-projective varieties. Moreover, the image of the exceptional divisors of $X'\to X$ under $f$ is not surjective on $Z$ since codim$(Sing(X)/X)\geq 2$, hence the horizontal parts of $B'$ and $B$ are the same.
        
        Consider $X_\eta$ the generic fiber of $X/Z$, then it is projective normal, Fano type and Gorenstein as $X$ is smooth. Suppose $\pi:Y\to X_\eta$ is a normalization of an irreducible component of $X_{\bar\eta}$, then by the behavior of canonical divisor under inseparable base changes \cite[1.1]{patakfalvi2018singularitiesgeneralfiberslmmp}, $K_Y+(p-1)C\sim\pi^*K_{X_\eta}$ where $C$ is an integral effective divisor (not 0 if $X_\eta$ is not geometrically normal), hence $K_Y\sim-(p-1)C+\pi^*K_{X_\eta}$ is anti-ample since deg$(-K_{X_\eta})>0$. Moreover $-2\leq$deg$(K_Y)\leq$deg$((1-p)C)\leq 1-p$, which implies $p\leq 3$, so we choose $p_0>3$.

        Since $X/Z$ admits normal geometric generic fiber, in the sense that general fiber $F$ is a projective smooth curve. Consider $(F,B_F)$, if a component $D$ of $B$ is vertical, then its contribution in $B_F=B|_F$ is $0$ as $F$ is a general fiber. So $B|_F=B^{hor}|_F$ and we only need to consider the horizontal part which impacts the singularity of the general fiber. For a horizontal component $D\in Supp(B)$, we can move the general fiber such that $D|_F$ has no multiplicity except the inseparable part. 
        
        Consider the surjection $D^\nu\to Z$ with inseparable degree $p^{k_D}$ of the field extension $K(D^\nu)/K(Z)$, then by the basic intersection theory, $D|_F=\sum p^{k_D} D_i$ where $D_i$ are reduced divisors. Hence, set $B^{hor}=\sum a_jD_j$ where $a_j\in \Phi$, we have $0\sim_\mathbb Q (K_X+B)|_F=K_F+B^{hor}_F=K_F+\sum_j\sum_ip^{k_{D_j}}a_jD_{j,i}$. If some $k_{D_j}>0$, then calculating the degree on both sides, one sees that $0\geq \text{deg}(K_F)+p\min(\Phi)\geq-2+p\min(\Phi)$, hence $p\leq \frac{2}{\min(\Phi)}$. Set $p_0\geq\frac{2}{\min(\Phi)}$ there is no inseparable horizontal part and the general fibers are lc, by \cite{Witaszek2017OnTC} the canonical bundle formula holds in this case. 
        
        In fact,  we have proved that the geometric generic fibre $X_{\bar{\eta}}$ is a smooth rational curve since it is normal and has a canonical divisor with a negative degree. Moreover we have proved that the horizontal part of $B$ is separable over $Z$, hence they are given by a union of rational sections of $X/Z$. We can use the arguments in the proof of \cite[3.1]{Witaszek2017OnTC} to see that $M_Z$ is semi-ample outside a codimension-$2$ closed subset, hence semi-ample since $Z$ is a surface. 
        
        When dim $Z=1$, the general fibers are now surfaces, and from now on one assumes $p>5$.  As $X\to Z$ is a contraction, we have that the generic fiber $X_\eta$ is defined on $\eta=k(Z)$ and $H^0(X,\mathcal O_X)=\eta$. Firstly, after running $K_X+(1+\upsilon)B\sim_\mathbb Q\upsilon B$-MMP on $(X,B)/Z$, we end with a good minimal model $(X',B')$ by abundance and LMMP and $(X,(1+\upsilon)B)$ is klt with $\upsilon$ small enough. Suppose $(X',B')\to (Y,B_Y)/Z$ is the semiample fibration, then one sees that $h:(X,B)\to (Y,B_Y)/Z$ is crepant and $-K_Y\sim_\mathbb QB_Y/Z$ is ample$/Z$ and $B_Y=h_*B$. Since $h:(X,B)\to (Y,B_Y)/Z$ is crepant, the construction of the canonical bundle formula is compatible and one can assume that $B$ is ample$/Z$. Now $(X_\eta,B_\eta)$ is $\epsilon=\frac{1}{I}$-lc and $X_\eta$ is a del Pezzo surface, hence $X_\eta$ falls in a bounded set $\mathcal P/Spec\mathbb Z$ by \cite{bernasconi2024boundinggeometricallyintegraldel}.

        Now there is a very ample divisor $H$ on $X_\eta$ with $-K_X\cdot H=B\cdot H$ bounded and $H^2$ bounded by some natural number $M$. Moreover $H$ induces a closed immersion $X_\eta\to \mathbb P^N_\eta$, hence its base change to algebraic closure $\bar H$ will also induce a closed immersion $X_{\bar\eta}\to \mathbb P^N_{\bar\eta}$, which implies that the geometric generic fiber also falls in a bounded set over $Spec\mathbb Z$. Consider the $\eta$-non-smooth (closed since $X_\eta$ is geometrically normal) points and on $X_\eta$ and the non-snc points of $B$, we call them $x_i$'s. As $X_\eta$ is a klt surface, we see $X_\eta$ is $\mathbb Q$-factorial and the Cartier index of $K_{X_\eta}$ and $B$ near $x_i$ is bounded by some natural number $N$ and, moreover, the defining equations of $X_\eta$ have bounded degrees $N$ since $(X_\eta,B_\eta)/k$ falls in a bounded set$/Spec\mathbb Z$. Thus, for the closed points $x_i$, $\kappa(x_i)/\eta$ has a degree bounded by $N$. Hence, for $p>N$, $\kappa(x_i)/k$ is separable. Now we are going to show that the pair $(X_\eta,B_\eta)$ is geometrically klt if $B_\eta$ is geometrically reduced.

        We denote $X$ to be $X_\eta$ for convenience in the following two paragraphs, the following arguments are similar to those in \cite{sato2024generalhyperplanesectionslog}. In fact, after shrinking $X$, we only need to consider a surface singularity $(x\in X,B)$. If $x\in X$ is smooth and $B$ is snc near $x$, since $\kappa(E)/k$ is smooth for each component $E$ of $B$, then the components of $B$ and the intersections are smooth over $\kappa(x)$, and therefore $(x\in X,B)$ is geometrically log smooth and the discrepancies will remain the same after base change to geometric case. So we may assume that $(x\in X,B)$ is an isolated (geometric) surface singularity. Pick a log resolution of singularity $f:Y\to X$ near $x$, say Exc$(x)=\sum\limits_iE_i$ is the sum of the exceptional divisors on $Y$, which is a scheme defined over $\kappa(x)$. Suppose $l$ is a purely inseparable field extension of $\kappa(x)$, then for any irreducible $\kappa(x)$-scheme $X$, $X\times_{\kappa(x)}l$ is homeomorphic to $X$ as topological spaces and hence also irreducible. Hence, we can take a finite separable field extension $k'$ of $\kappa(x)$ such that the irreducible components of $E_i\times_{\kappa(x)}k'$ are geometrically irreducible. Since $k'/\kappa$ is finite separable, and hence a single extension by a monic polynomial $g(t)\in \kappa[t]$. Suppose $X=Spec(R)$ with $\kappa=R/\mathfrak m$, take a monic lift $\tilde g(t)\in R[t]$ of $g(t)$, we have that $X'=Spec(R[t]/\tilde g(t))\to Spec(R)=X$ is finite étale and surjective, with the corresponding $x'\to x$ realizing $\kappa\to k'$.  Then $Y'=Y\times_XX'$ is a minimal log resolution of singularity by étale descent, whose components of exceptional locus are exactly such components of $E_i\times_\kappa k'$ and hence geometrically irreducible. Also $X'$ is klt with the same discrepancies on certain components by étale descent, and hence has rational singularities. Hence, we have $g(E_i)=0$ and by the classification of the dual graph \cite[A.3]{sato2024generalhyperplanesectionslog}, we have $\dim_{\kappa}H^0(E_i,\mathcal O_{E_i})\leq 4$. Hence $K=H^0(E,\mathcal O_E)$ is separable over $\kappa(x)$ and hence over $k$ if $p>3$, which means that $E$ is geometrically reduced as $E$ is reduced$/K$ which can be realized as a conic in $\mathbb P_{K}^2$ and no purely inseparable field extension $L/K$ will make $E_L$ to be a $p$-power multiple of the certain divisor since $E_L=\mathcal O_{\mathbb P_L^2}(2)$ and $p>2$.

        Now $E_i$ are geometrically integral and their base change to $\bar \kappa(x)$ are integral curves with arithmetic genus $0$, hence $\mathbb P^1$, which are smooth. Moreover, we have $B_i$'s that are geometrically reduced (and hence integral by a finite étale base change) by the assumption, hence we have the exceptional locus together with the strict transformation of $B$ is snc and geometrically integral with smooth exceptional divisors. Consider each $B_i$ which is a geometrically integral regular curve over $k$, suppose it is not normal, with the normalization of the geometric model $C_i$, then we have $K_{C_i}+(p-1)D\sim_\mathbb Q K_{B_i}|_{C_i}$. Count the degree, since $B_i$ is Gorenstein and by Riemann-Roch, we have $2g_C-2+(p-1)deg(D)=2g_{B_i}-2$ and hence $g_{B_i}\geq\frac{p-1}{2}$. However, in our case $HB_i$ is bounded, which means $g_{B_i}$ is bounded by some $G$, hence this could not happen if $p>2G+2$. Hence, all $E_i$'s and $B_i$'s are smooth$/k$. View $Exc(x)+\tilde B$ as a scheme over $k$, then all the extension degrees of $h^0$ of the components are bounded, hence when $p$ is large enough, all intersections $E_i\cap E_j$ or $E_i\cap B_j$ are smooth over $x\in X$. Hence $(Y,Exc(x)+\tilde B)$  is snc after arbitrary field base change and we see that all the discrepancies would remain the same after base change to the algebraic closure.
        
        Now, return to our initial fibration case, we are going to show that $B_\eta$ is geometrically reduced. Suppose $B=B^{hor}+B^{ver}=\sum a_jD_j+B^{ver}$, then the vertical part will make no matter on the singularity of the fibers. After restriction to a general fiber we have $0\sim_\mathbb Q (K_X+B)|_F=K_F+\sum_ja_{j}p^{k_{D_j}}D_{j,F}^{red}$ where $p^{k_{D_j}}$ is the inseparable degree of (the Stein factorization of) the map $D_j\to Z$. Then since $(F,B_F)$ falls into a bounded set and $-\bar H\cdot K_F\leq M$, we have $\bar H\cdot D_{j,F}^{red}\geq \frac{1}{N}$ as $NB$ is Cartier near the generic fiber by boundedness, if there is some $k_{D_j}>0$, then $M\geq -K_F\cdot \bar H\geq pa_jD_{j,F}^{red}\cdot\bar H\geq \frac{p\min(\Phi)}{N}$, hence a contradiction when $p>\frac{NM}{\min(\Phi)}$ or just $p>NMI$. Thus when $p$ is large enough, we see there is no inseparable part both in the map from the exceptional divisors and the components of $B$ to $Z$, in particular the log discrpancies are kept and hence $(F,B_F)$ is geometrically lc and by \cite[0.2]{benozzocanonicalbundleformula}, the canonical bundle formula holds in this case.

        For the assertion (3), we only need to treat the case when $B$ is relatively ample by running $K_X+(1+\upsilon)B$-MMP on $X/Z$. Here the generic fiber is an $\epsilon$-lc del Pezzo surface and hence falls in a bounded set. Since the coefficients of $B$ falls in a finite set, $(X_\eta,B_\eta)$ falls in a log bounded set. So there is a uniform $I\in\mathbb N$ such that $I(K_{X_\eta}+B_\eta)$ is Cartier. Then the assertion (3) follows from the assertion (2).

\end{proof}
\begin{remark} 
    The method in the proof of Theorem 3.3 is difficult to generalize to fibrations of Calabi-Yau type or of general type in positive characteristic since it relies on positivity (or negativity) and  boundedness of the geometric generic fiber to control the singularity of general fibers. 

    One usually expects that $\epsilon$-lc del Pezzo surfaces are geometrically normal for sufficiently large characteristics like the regular case, c.f. \cite[1.5]{patakfalvi2018singularitiesgeneralfiberslmmp}, which essentially removes the condition that we have normal general fibers and makes the proposition more general.  However, the author does not yet know if this is true. But another interesting observation is that if $p$ is large enough with exactly the same conditions but $X_\eta$ is not geometrically normal and geometrically $\mathbb Q$-factorial, then the normalization of general fibers $(X_g,B_g)$ cannot be lc hence the canonical bundle formula will seldom apply (in particular, the general fibers cannot admit klt singularities). In fact, if $X$ is a geometrically reduced variety$/k$ which falls into a bounded set$/Spec\mathbb Z$, consider the normalization $Y$ of $X_{\bar k}$, then we have $K_Y+(p-1)C+B_Y\sim_\mathbb Q(K_X+B)|_Y$.  Since $(X,B)$ falls into a bounded set$/Spec\mathbb Z$ and $Y$ is $\mathbb Q$-factorial, the Cartier index of $Y$ is bounded by a uniform $N$. So, if $Y$ is lc but is not isomorphic to $X_{\bar k}$, then $K_Y+B_Y$ is not nef as $C> 0$. By the cone theorem there is some rational curve $E$ such that $0<-(K_Y+B_Y)\cdot E\leq 2\text{dim}Y=4$. However, $-(K_Y+B_Y)\cdot E=(p-1)C\cdot E\geq \frac{p-1}{N^2}$, which is a contradiction when $p>4N^2+1$. 
\end{remark}
With some characteristic free arguments on the geometry of non-klt centers, we can prove an adjunction formula on non-klt centers for threefolds.
\begin{theorem}
    $(X,B)$ is a projective klt pair of dimension $3$ over an algebraically closed field $k$, $G\subset X$ a subvariety with normalization $F$, $X$ is $\mathbb Q$-factorial near $G$, $\Delta\geq 0$ a $\mathbb Q$-Cartier divisor and $(X,B+\Delta)$ is lc near $G$ and there is a unique non-klt place of this pair whose center is $G$. Let $\Phi$ be a finite set of rationals in $(0,1]$ such that $B\in\Phi$. Then there are a $p_0=p_0(\Phi)\in\mathbb N$ and a DCC set $\Psi=\Psi(\Phi)$ such that for all char$(k)>p_0$, there is a $\mathbb Q$-divisor $\Theta_F\in[0,1]$ such that the following hold:
    \begin{enumerate}
        \item There is $P_F$ on $F$ such that 
        \[K_F+\Theta_F+P_F\sim_\mathbb Q (K_X+B+\Delta)|_F,\]
        here $\Theta_F$ is well-defined and $P_F$ is determined up to $\mathbb Q$-linear equivalence and $P_F$ is pseudo-effective.
        \item $\Theta_F\in\Psi$.
    \end{enumerate}
\end{theorem}
\begin{proof}
    If $G$ is a divisor, then the results follow from the divisorial adjunction formula. If $G$ is a point, then this is trivial. From now on, we assume dim$G= 1$. 

    Let $\phi:W\to X$ be a log resolution of $(X,B+\Delta)$ that extracts the place above $G$. Define $\Gamma=(B+\Delta)^{<1}+\text{Supp}(B+\Delta)^{\geq 1}$ and $\Gamma_W$ as the sum of the strict transformation $\Gamma^\sim$ and the reduced exceptional divisor of $\phi$. Running an MMP on $(W,\Gamma_W)/X$ we reach a model $\psi:(Y,\Gamma_Y)\to X$ such that $K_Y+\Gamma_Y$ is a limit of movable divisors over $X$. Now we set $N_Y:=\psi^*(K_X+B+\Delta)-(K_Y+\Gamma_Y)$. As $\psi_*(N_Y)=N:=B+\Delta-\Gamma\geq 0$, and $-N_Y$ is a limit of movable divisors$/X$, so by the general negativity lemma, $N_Y$ is effective. Let $U\subset X$ be the set of points where $(X,B+\Delta)$ is lc, then $N_Y=0$ over $U$ and $(Y,\Gamma_Y)$ is a $\mathbb Q$-factorial dlt model of $(X,B+\Delta)$ over $U$ by construction and $U$ contains $G$. Then by \cite[2.33]{birkar2019antipluricanonicalsystemsfanovarieties}, no non-klt center contains $G$ except $G$, so there is only one component $S$ of $\lfloor \Gamma\rfloor$ mapping onto $G$, which has coefficient $1$ so the image of $N_Y$ does not contain $G$. Again, by \cite[2.33]{birkar2019antipluricanonicalsystemsfanovarieties}, the natural morphism $h:S\to F$ predicted by the property of normalization is a contraction. By divisorial adjunction we have 
    \[K_S+\Gamma_S+N_S=(K_Y+\Gamma_Y+N_Y)|_S\sim_\mathbb R 0/F\]
    where $N_S=N_Y|_S$ is vertical over $F$ since the generic point is not mapped. Moreover $(S,\Gamma_S+N_S)$ is lc by inversion of adjunction. We consider $\Sigma_Y=B^\sim+exc(\psi)\in\Phi\cup \{1\}$ and set $K_S+\Sigma_S\sim_\mathbb Q(K_Y+\Sigma_Y)|_S$, one sees that $\Sigma_S\in \Phi'$ for some DCC set $\Phi'$ depending only on $\Phi$ by divisorial adjunction. By definition $\Sigma_S\leq\Gamma_S+N_S$, as $S\to F$ is a morphism from a surface to a curve, so by Theorem 3.3 the canonical bundle formula applies for $(S,\Gamma_S+N_S)$ for all char$(k)>p_0$ with some $p_0=p_0(\Phi')=p_0(\Phi)$. We have $K_S+\Gamma_S+N_S\sim_\mathbb Qh^*(K_F+D_F+M_F)$ where $D_F$ is the discriminant part and $M_F$ is pseudo-effective moduli divisor. Since the adjunction for fiber space is determined up to $\mathbb Q$-linear equivalence, so $K_F+D_F+M_F\sim_\mathbb Q(K_X+B+\Delta)|_F$ by chasing the pullback diagram:
    \[\begin{tikzcd}
	S & Y \\
	F & X
	\arrow[from=1-1, to=1-2]
	\arrow["h", from=1-1, to=2-1]
	\arrow[from=2-1, to=2-2]
	\arrow["\psi", from=1-2, to=2-2]
\end{tikzcd}\]
    Now let $\Theta_F$ to be discriminant part of $(S,\Sigma_S)$ as usual. Consider $P_F\sim_\mathbb RD_F-\Theta_F+M_F$, which is pseudo-effective as $D_F\geq\Theta_F$ by definition. As a result, the coefficients of $\Theta_F$ belong to
    \[\Psi:=\{1-a|a\in LCT_{2}(\Phi')\}\cup \{1\}\]
    where $LCT_{2}(\Phi')$ stands for the set of all lc thresholds of integral effective divisors with respect to pairs $(S,\Sigma)$ of dimension $2$ such that $\Sigma\in \Phi'$, which is an ACC set by ACC of lc thresholds in dimension 2. So $\Psi$ is also a DCC set.
\end{proof}
\begin{theorem}
    Assume $X,G,F,\Delta,B,\Phi,p_0,\Psi,\Theta_F,P_F$ defined as above, $M\geq 0$ is a $\mathbb Q$-Cartier divisor on $X$ with coefficients $\geq 1$ and $G$ is not a subset of $\text{Supp}M$, then for every component $D$ of $M_F:=M|_F$, we have $\mu_D(\Theta_F+M_F)\geq 1$.
\end{theorem}
\begin{proof}
    Set $M_Y=M|_Y$ and $\Sigma'_Y=\Sigma_Y+M_Y$, we have $K_S+\Sigma'_S\sim_\mathbb Q(K_Y+\Sigma'_Y)|_S$. Then $M_S=M|_S=h^*M_F$ is $h$-trivial and hence $K_S+\Sigma'_S+\Gamma_S-\Sigma_S+N_S\sim_\mathbb Q0/F$. Since $M$ does not pass through the generic point of $G$, $(S,\Sigma'_S+\Gamma_S-\Sigma_S+N_S)$ is lc over the generic point of $G$. So by Theorem 3.3, the canonical bundle formula applies for $(S,\Sigma'_S+\Gamma_S-\Sigma_S+N_S)$. Let $\Theta'_F$ be the discriminant part of $(S,\Sigma'_S)$, then we have for every component $D$ of $M_F$, $\mu_D(\Theta_F+M_F)=\mu_D(\Theta'_F)=1-\text{lct}_\eta(h^*D,S,\Sigma'_S)$.  For any $D$ component of $M_F$, $h^*D$'s components appear as components of $\lfloor\Sigma_S+M_S\rfloor=\lfloor \Sigma_S'\rfloor$  by divisorial adjunction, for $M\geq 1$ and $S$ is not a component of $M_Y$. So $\text{lct}_\eta(h^*D,S,\Sigma'_S)\leq 0$ as desired.
\end{proof}
The following is a proposition of common use (in a special case) when studying Fano varieties:
\begin{corollary}
    Let $X$ be a threefold of Fano type, $X\to Z$ be a contraction and dim $Z>0$ with normal general fibers. Suppose there is an ample $\mathbb Q$-divisor $A$ on $Z$ and some given $I>0$ such that $I(K_X+B)\sim -If^*A$ are Cartier divisors for some klt pair $(X,B)$, then there is an $p_0=p_0(I)$ such that for every $p>p_0$, such $Z$ is of Fano type.
\end{corollary}
\begin{proof}
    If dim $Z=3$, then this follows from our arguments in Chapter 2. If dim$Z\leq 2$, then $I(K_X+B)$ is Cartier and $K_X+B\sim_\mathbb Q 0/Z$, hence the canonical bundle formula applies by Theorem 3.3. Taking a high resolution $X'\to Z'$ of $X\to Z$ we may assume $X$ and $Z$ to be smooth since the construction of the canonical bundle formula is compatible with crepant birational maps\cite[3.4]{birkar2019antipluricanonicalsystemsfanovarieties}. Consider the coefficients of $B_Z$, each has the form $1-lct_\eta(X,B,f^*D)$, and therefore $<1$ as $(X,B)$ is klt. So $(Z,B_Z)$ is klt, and moreover $K_Z+B_Z+M_Z+A\sim_\mathbb Q0$ by the condition. As $M_Z$ is nef, one sees that $M_Z+A$ is ample and there is some $A'\sim_\mathbb Q M_Z+A$ such that $(Z,B_Z+A')$ is klt. Hence $Z$ is of Fano type by definition.
\end{proof}
There is another kind of adjunction in positive characteristic concerning F-singularities, we usually call it F-adjunctions, which is obtained by the technology of map pairs:  
\begin{theorem}[\cite{schwede2009fadjunction} 5.2,6.17]
    Suppose that $X$ is an integral separated normal F-finite noetherian scheme essentially of finite type over an F-finite field of characteristic $p>0$. Further suppose that $\Delta$ is an effective $\mathbb Q$-divisor on $X$ such that $K_X + \Delta$ is $\mathbb Q$-Cartier with index not divisible by $p$. Let $W\subset X$ be a closed subscheme (or the normalization of a closed subscheme)that satisfies the following properties:
\begin{enumerate}
    \item $W$ is integral and normal.
    \item $(X,\Delta)$ is sharply F-pure at the generic point of $W$.
    \item The ideal sheaf of $W$ is locally a center of sharp F-purity for $(X,\Delta)$.
\end{enumerate}
Then there exists a canonically determined effective divisor $\Delta_W$ on $W$ satisfying the following properties:
\begin{enumerate}
    \item $(K_W +\Delta_W) \sim_\mathbb Q (K_X +\Delta)|_W$.
    \item Furthermore, if $(p^e-1)(K_X + \Delta)$ is Cartier then $(p^e-1)(K_W + \Delta_W )$ is Cartier and $(p^e-1)\Delta_W$ is integral.
    \item $W$ is minimal among the centers of sharp F-purity for $(X,\Delta)$, with respect to the containment of topological spaces (in other words, the ideal sheaf of $W$ is of maximal height as a center of sharp F-purity), if and only if $(W, \Delta_W )$ is strongly F-regular.
    \item There is a natural bijection between the centers of sharp F-purity of $(W,\Delta_W)$ and the centers of sharp F-purity of $(X, \Delta)$ that are properly contained in $W$ as topological spaces.
\end{enumerate}
\end{theorem}
    The theorem predicts that a similar phenomenon as general adjunctions for MMP singularities would happen under the setting of the F-singularities. F-adjunctions would be useful for us to construct sections in the proof of effective birationality in the next chapter.
\section{A weak version of effective birationality of Fano varieties}
 To see birationality of some linear system, the first observation is that F-pure centers of a single point (bad singularities) can give rise to sections separating points:
\begin{theoremdef}[F-potentially birational divisors]
    A big $\mathbb Q$-Cartier $\mathbb Q$-divisor $D$ on a variety $X$ (admitting a resolution of singularity) is called F-potentially birational, if for any two general points $x$ and $y$ on $X$, possibly switching $x$ and $y$, one may find some $0\leq\Delta\sim_\mathbb Q (1-\epsilon)D$ such that $(X,\Delta)$ is F-pure near $x$ but not F-regular at $y$ and $\{x\}$ is an isolated F-pure center, and the Cartier index of $K_X+\Delta$ is not divisible by $p$ near $x$. 

    Suppose $D$ is an F-potentially birational divisor on $X$, then $\phi_{K_X+\lceil D\rceil}$ defines a birational map.
\end{theoremdef}
\begin{proof}
     Since $D$ is big, we may write $D\sim_\mathbb QA+B$ for some ample divisor $A$ and effective divisor $B$. Set $D_\epsilon=\epsilon B+\lceil D\rceil -D$ for some $\epsilon\in\mathbb Z_{(p)}$ with $0<\epsilon\ll1$, then the following argument with the same strategy to do some tie breaking as in \cite{wang2018characterizationabelianvarietieslog} and\cite{tanaka2017semiampleperturbationslogcanonical} will give some $0\leq D'_\epsilon\sim_\mathbb Q D_\epsilon$ such that $(X,\Delta+D'_\epsilon)$ is strongly F-regular in a punctured neighbourhood of $x$ for $x\notin Bs(|D_\epsilon|_\mathbb Q)$:
    
    After shrinking near $x$, we may assume $(X,\Delta)$ is sharply F-pure. Set $W_i^*(X,\Delta)$ to be the union of non-F-regular centers of dimension $\leq i$, set $W_i(X,\Delta)=W_i^*(X,\Delta)-W_{i-1}^*(X,\Delta)$. Then every component $W$ of $W_i(X,\Delta)$ is a minimal F-pure center of the sharply F-pure pair $(X-(\bar W-W),W)$, hence normal by \cite{schwede2009centersfpurity}. Set $X=U\cup\bigcup_{i,k}W_i^k$ a stratification along irreducible F-pure centers, then for each $W_i^k$, we have an F-adjunction formula $K_{W_i^k}+\Delta_{W_i^k}=(K_X+\Delta)|_{W_i^k}$ by Theorem 3.8, such that $(W_i^k,\Delta_{W_i^k})$ is a strongly F-regular pair. Also, $(U,\Delta_{U})$ is strongly F-regular. Take $m$ not divisible by $p$, which is large enough such that $mD_\epsilon$ is Cartier and $Bs(|mD_\epsilon|)$ is the stable base locus of $D_\epsilon$.

    One sees if we take $\Gamma\in |mD_\epsilon|$ does not contain any F-pure center, there is $e_0$ such that for any $e\geq e_0$, all $(W_i^k,\Delta_{W_i^k}+\frac{d}{m(p^e-1)}\Gamma|_{W_i^k})$ and $(U,\Delta_{U}+\frac{d}{m(p^e-1)}\Gamma|_{U})$ are strongly F-regular. Now consider the map $W_i^k\times |mD_\epsilon|^d\to X\times |mD_\epsilon|^d\to |mD_\epsilon|^d$, applying \cite[4.19]{patakfalvi2017fsingularitiesfamilies}, one sees that there is an open neighbourhood $V$ of $\Gamma\in|mD_\epsilon|$ such that for any $(D_1,D_2,\cdots,D_d)\in V^d\subset |mD_\epsilon|^d$, $(W_i^k,\Delta_{W_i^k}+\frac{1}{m(p^e-1)}(\sum\limits_{i=1}^dD_i)|_{W_i^k})$ and $(U,\Delta_{U}+\frac{1}{m(p^e-1)}(\sum\limits_{i=1}^dD_i)|_{U})$ they are all strongly F-regular. Now take $D'_\epsilon=\frac{1}{m(p^e-1)}\sum\limits_{i=1}^{p^e-1}D_i$ where $D_i\in V$ are general elements for $i\in \{1,2,\cdots,p^e-1\}$ such that the intersection of any distinct $d+1$ $D_i$'s  is the base locus $Bs(|mD_\epsilon|)$. Then near each point $y\notin Bs(|D_\epsilon|_\mathbb Q)$, say $y\in W\in\{U,W_i^k\}$, there are at most $d$ of $D_i$'s  passing through $y$, so $(W,\Delta_W+\frac{1}{m(p^e-1)}(\sum\limits_{i=1}^{p^e-1}D_i)|_W)$ is strongly F-regular near $y$. Since $x\notin Bs(|mD_\epsilon|)$, after proper shrinking, $(X,\Delta+D'_\epsilon)$ is again F-pure near $x$ and $\{x\}$ is an isolated F-pure center.  
    
    Now come back to our problem, consider $\lceil D\rceil-(\Delta+D'_\epsilon)\sim_\mathbb Q \epsilon A$, which is ample. Set $Z\subset X$ to be the (normalized) non-F-regular center of $(X,\Delta+D'_\epsilon)$, with the defining ideal $I_Z=\tau(X,\Delta+D'_\epsilon)$ the test ideal (recall that the test ideal $\tau(X,B)$ of the pair $(X,B)$ is the smallest nonzero ideal $J\subset \mathcal O_X$ such that $\phi(F^e_*J)\subset J$ for any local section $\phi\in\mathscr{H}om_{\mathcal O_X}(F^e_*\mathcal O_X(\lceil (p^e-1)B\rceil),\mathcal O_X)$ and all $e>0$). Consider the following exact sequence:
    \[\begin{aligned}0&\to H^0(X,\mathcal O_X(K_X+\lceil D\rceil))\otimes I_Z)\to H^0(X,\mathcal O_X(K_X+\lceil D\rceil))\xrightarrow{\gamma} H^0(Z,(K_X+\lceil D\rceil)|_Z)\\&\to H^1(X,\mathcal O_X(K_X+\lceil D\rceil))\otimes I_Z)\to\cdots
    \end{aligned}\]
    Set $K_Z+D_Z=(K_X+\lceil D\rceil)|_Z$ by the general F-adjunction, by F-version of Nadel vanishing using test ideals \cite[5.5]{blickle2014fsingularitiesalterations}, one sees that the image of $\gamma$ contains 
    \[T^0(Z,D_Z)=\bigcap\limits_{f:Y\to Z \text{ finite dominant maps}}Image(H^0(Y,\lceil K_Y+f^*D_Z\rceil)\xrightarrow {Tr_f} H^0(Z,K_Z+D_Z))\]
    where $Y$ is a normal equidimensional variety and $Tr_f$ is the trace map (of the lowest cohomology) induced by $Tr_f:Rf_*\omega^\cdot_Y\to \omega^\cdot_Z$, which is dual to the structure map $\mathcal O_Z\to f_*\mathcal O_Y$. Since the point $\{x\}\subset Z$ is a connected component and $y\in Z$, there is therefore a section in $H^0(X,\mathcal O_X(K_X+\lceil D\rceil))$ which is not zero at $x$ but vanishes at $y$. Similarly $y$ can be taken very generally so that there is a section in $H^0(X,\mathcal O_X(K_X+\lceil D\rceil))$ which does not vanish at $y$. Hence, the sections of $K_X+\lceil D\rceil$ separate points, which means that $\phi_{K_X+\lceil D\rceil}$ is birational.
\end{proof}
\begin{remark}
    If $(X,\Delta)$ is lc at $x$ and not klt at $y$, where $x,y$ are in the smooth locus of $X$ and away from the stable base locus of $\Delta$, then the pair is at best sharply F-pure near $x,y$. One can always find $\Delta'\sim_\mathbb Q\Delta$ such that $(X,\Delta')$ is strongly F-regular at $x$ and $y$. So there is $\theta\in[0,1]$ such that $(X,\Delta_\theta=\theta\Delta'+(1-\theta)\Delta)$ is exactly F-pure at $x$ but not F-regular at $y$ after some possible switching of $x$ and $y$ by the following Lemma 4.2. However, these do not mean that a potentially birational divisor is an F-potentially birational divisor, since the F-pure center of $(X,\Delta_\theta)$ near $x$ may not be $\{x\}$.

    Indeed combine the statements above and \cite[7.1]{hacon2012acclogcanonicalthresholds} and some tie breaking arguments, we have if $D$ is an ample $\mathbb Q$-divisor with $vol(D)\geq (2d)^d$, then there is a bounded covering family $\{G_i\}$ of subvarieties of $X$ such that for any $x,y\in X$ general closed points, there is some $\Delta\sim_\mathbb Q D$ such that $(X,\Delta)$ is F-pure at $x$ with a unique F-pure center $G_i$ containing $x$ and not F-regular at $y$.
\end{remark}
The following statements are crucial for cutting the F-pure centers.
\begin{lemma}
    Assume $(X,B)$ is a strongly F-regular pair and $x\in X$ a closed point. Suppose $D$ is an effective $\mathbb Z_{(p)}$-Cartier $\mathbb Q$-divisor on $X$ such that $(X,B+D)$ is sharply F-pure but not strongly F-regular at $x$, with the non-F-regular center $Z\cup Z'$, where $Z$ is the unique component containing $x$ with the generic point $\xi$. Assume $\Delta$ is an effective $\mathbb Z_{(p)}$-Cartier $\mathbb Q$-divisor that does not pass through $\xi$, then $(X,B+D+\Delta)$ is sharply F-pure at $\xi$. Moreover, for any $0<\epsilon\in\mathbb Z_{p}$, we have that $(X,B+D+\epsilon D)$ is not sharply F-pure at $\xi$. 
\end{lemma}
\begin{proof}
    We may assume $X=Spec(R)$ since the problem is local. Consider the map of $R_\xi$-modules:
    \[R_\xi\to F^e_*R_\xi(\lceil (p^e-1)(B+D)\rceil)\xrightarrow{F^e_*(i\otimes R_\xi(\lceil (p^e-1)(B+D)\rceil))} F^e_*R_\xi(\lceil (p^e-1)(B+D+\Delta)\rceil)\]
    Where $(p^e-1)\Delta$ is Cartier. The first morphism splits by sharp F-purity and the second is an isomorphism since $i:R_\xi\to R_\xi((p^e-1)\Delta)$ is an isomorphism. By a similar reason for the strongly F-regular pair $(X,B)$, one sees that $D$ passes through $\xi$. If $(X,B+D+\epsilon D)$ is sharply F-pure near $\xi$, then by \cite[2.15]{Schwede_2013}, $(X,B+D)$ is strongly F-regular near $\xi$, which is a contradiction.
\end{proof}
\begin{lemma}
    Assume $(X,B)$ is a strongly F-regular pair and $x\in X$ a general closed point. Suppose $D$ is an effective $\mathbb Z_{(p)}$-Cartier $\mathbb Q$-divisor on $X$ such that $(X,B+D)$ is sharply F-pure but not strongly F-regular at $x$, with the non-F-regular center $Z\cup Z'$, where $Z$ is the unique component containing $x$ with the generic point $\xi$ and dimension $k$. Assume $H$ is an ample $\mathbb Z_{(p)}$-Cartier $\mathbb Q$-divisor with $H^k\cdot Z>k^k$, then there is an effective $\mathbb Q$-divisor $A\sim_\mathbb Q H$ and $0<\delta\ll1$ and $0<c<1$ such that:
    \begin{enumerate}
        \item $(X,B+(1-\delta)D+cA)$ is sharply F-pure at $x$ while not strongly F-regular.
        \item $(X,B+(1-\delta)D+cA)$ has non-F-regular center $Z_1\cup Z_1'$, where $Z_1$ is the unique component containing $x$ whose dimension is smaller than that of $Z$.
    \end{enumerate}
    Moreover, if $(X,B+D)$ is not strongly F-regular at $y$ for some general point $y$, then we can choose $(X,B+(1-\delta)D+cA)$ such that it is not strongly F-regular at $y$.
\end{lemma}
\begin{proof}
    Take $m>>1$ large enough such that $mH$ is Cartier and very ample enough such that $\mathcal O_X(mH)\otimes I_Z$ is globally generated and has vanished higher cohomologies by Serre's vanishing. As a result, $H^0(X,\mathcal O_X(mH))\to H^0(Z,\mathcal O_Z(mH|_Z))$ is surjective. Therefore, if we take a suitable $F_Z\sim_\mathbb Q mH$ on $Z$ such that $\text{mult}_x(F_Z)>mk$, there is $F$ on $X$ such that $F|_Z=F_Z$. Consider the linear subsystem $|F|$ of $|mH|$ consisting of those $F'$ such that $F'$ contains $Z$ or $F'|_Z=F_Z$, it is a closed subset of $|mH|$. We can take some $F''\in |F|$ such that it contains no F-pure centers of $(X,B+D)$, we see that $|F|$ is free on $X-(Z\cup Z')$. By the same cutting arguments as in the proof of Theorem 4.1, there is an open neighbourhood $V$ of $F''$ in $|F|$ such that we can take $F_1,F_2,\cdots,F_{p^e-1}\in |F|$ for $e$ large enough (possibly after enlarging $m$), which is simple normal crossing outside $Z\cup Z'$, and $(X-Z,B+D+\frac{1}{m(p^e-1)}\sum\limits_{i=1}^{p^e-1}F_i)$ is sharply F-pure and shares the same F-pure center with $(X-Z,B+D)$. 

    Now, since $\frac{1}{m(p^e-1)}\sum\limits_{i=1}^{p^e-1}F_i|_Z=\frac{1}{m}F_Z$, while $(Z,B_Z+D_Z+\frac{1}{m}F)$ is not lc at $x$ by counting multiplicity, hence $(Z,B_Z+D_Z+\frac{1}{m}F_Z)$ is not F-pure near $x$. By F-adjunction, one sees that $(X,B+D+\frac{1}{m(p^e-1)}\sum\limits_{i=1}^{p^e-1}F_i)$ is not F-pure near x. Since $\frac{1}{m(p^e-1)}\sum\limits_{i=1}^{p^e-1}F_i$ contains no F-pure center of $(X,B+D)$ by our choice,  $(X,B+D+\frac{1}{m(p^e-1)}\sum\limits_{i=1}^{p^e-1}F_i)$ is F-pure near $\xi$. Taking $0<\delta\ll1$, one sees that $(X,B+(1-\delta)D+\frac{1}{m(p^e-1)}\sum\limits_{i=1}^{p^e-1}F_i)$ is strongly F-regular near $\xi$ while not F-pure near $x$ according to Lemma 4.2, so there is some $0<c<1$ making $(X,B+(1-\delta)D+\frac{c}{m(p^e-1)}\sum\limits_{i=1}^{p^e-1}F_i)$ is exactly F-pure with a F-pure center $Z_1$ near $x$, where $\text{dim}$ $Z_1<\text{dim}$ $Z$, hence taking $A=\frac{1}{m(p^e-1)}\sum\limits_{i=1}^{p^e-1}F_i)$ and we are done.

    Finally, consider the pair $(X,B+(1-\delta)D+cA)$ near $y$. If $y\in Z_1$, then we are done. If $y\notin Z_1$ but $y\in Z$, we can run the same process above on $y$ (like we have done for $x$) and find $A'\sim_\mathbb Q H$ not passing $x$ and $c'$ so that $(X,B+(1-\delta)D+cA+c'A')$ is not F-regular near $y$. If $y\in Z'$, say in some component $J$, then $x\notin J$, take $A'\sim_\mathbb Q H$ containing $J$ that does not pass $x$ and suitable $c'$, we can see that $(X,B+(1-\delta)D+cA+c'A')$ is not F-regular near $y$ as well by freeness.
\end{proof}
The following statement is an analogue of effective birationality in characteristic $0$ case.
\begin{theorem}
    Let $d,I\in\mathbb N$ be natural numbers where $I$ is not a multiple of $p$, then there is $n,m\in\mathbb N$ depending only on $d,I$, such that if $X$ is a strongly F-regular weak Fano variety of dimension $d$ and $IK_X$ is a Cartier divisor, then $vol(-nK_X)>(2d)^d$ and $|-mK_X|$ defines a birational map.
\end{theorem}
\begin{proof}
    The rough idea of the proof is to make a covering family of F-pure centers and use induction by adjunction formulas.

    For the first assertion, one sees that $-IK_X$ is a nef and big Cartier divisor, hence $vol(-IK_X)=\frac{(-IK_X)^d}{d!}\geq\frac{1}{d!}$. If $n'>2d\sqrt[d]{d!}$, one can see $vol(-n'IK_X)>(2d)^d$ as desired.

    For the second assertion, we only prove the case where $X$ is Fano. Suppose $m$ is not bounded, where $m$ is the smallest natural number such that $|-mK_X|$ defines a birational map, one can find a sequence of such tuples $(X_i,m_i)$ where $m_i\to\infty$. Keep $n$ as in the first assertion, one sees by Remark 4.1, there is a bounded covering family $\{G_i^s\}_{s\in S}$ of $X_i$ such that for general points $x,y\in X_i$, there is a $0\leq\Delta_i\sim_\mathbb Q-(n+1)K_{X_i}$ such that $(X_i,\Delta_i)$ is sharply F-pure at $x$ with a unique (normalized) F-pure center $G_i^s$ for some $s\in S$ and not strongly F-regular at $y$. Set $d_i$ to be the dimension of general $G_i^s$, consider the limit inferior of $\{d_i\}$, say $k$. 
    
    If $k=0$, then consider some $X_i$ such that $d_i=0$, and one sees that $-2(n+1)K_{X_i}$ is F-potentially birational (c.f. \cite[2.3.4]{hacon2012birationalautomorphismsvarietiesgeneral}), hence $|K_{X_i}+\lceil-2(n+1)K_{X_i}\rceil|$ defines a birational map by Theorem 4.1. Hence $m_i\leq Mn$ for some big $M$.

    If $k> 0$. Consider the smallest $l_i(=l)$ such that $vol(-lnK_{X_i}|_{G_i})>k^k$, then by Lemma 4.3, there is a replacement $\Delta'_i\sim_\mathbb Q -4lnK_{X_i}$ of $\Delta_i$ such that the dimension of the certain F-pure center decreases. As a result, $X$ has a bounded covering family $\{G_i^{'s}\}|_{s\in S'}$ such that for general points $x,y\in X_i$, there is a $0\leq\Delta'_i\sim_\mathbb Q-4lnK_{X_i}$ such that $(X_i,\Delta'_i)$ is sharply F-pure at $x$ with a unique (normalized) F-pure center $G_i^{'s}$ for some $s\in S$ and not strongly F-regular at $y$, and  $d_i'$ decreases. By Theorem 3.8 $IK_X|_G$ is again Cartier, and hence $l$ is bounded. So the statement is proved by induction on $k$.

\end{proof}
\begin{corollary}
    Assume $X$ is $\epsilon$-strongly F-regular weak Fano variety of dimension $d$ and the Gorenstein index is not divisible by $p$, then there is a constant $m=m(d,\epsilon)$ such that $|-mK_X|$ defines a birational map.
\end{corollary}
\begin{proof}
    By \cite[4.8]{Bhatt_2017}, $[\frac{1}{\epsilon}+1]!K_X$ is Cartier, then apply Theorem 4.4 and we are done.
\end{proof}

\bibliographystyle{alpha}
\bibliography{cite}
\end{document}